\documentclass[12pt]{article}

\usepackage{diagbox}
\usepackage{mathtools}
\usepackage{bbm}
\usepackage{latexsym}
\usepackage{epsfig}
\usepackage{amsmath,amsthm,amssymb,enumerate,stmaryrd}

\usepackage[a-1b]{pdfx}
\usepackage{hyperref}

\usepackage{color}

\def\reals{\mathbb R}

\def\hd{\widehat{d}}

\newcounter{rot}%\addtocounter{rot}{1}, \therot

\allowdisplaybreaks

\def\td{d^*}
\def\cR{{\mathcal R}}
\def\nn{\nonumber}
\def\a{\alpha}   \def\D{\Delta}
\def\e{\varepsilon}    \def\g{\gamma}
  \def\k{\kappa}
\def\z{\zeta} \def\th{\theta}    
\def\La{\Lambda}  \def\n{\nu} \def\p{\pi}
\def\r{\rho}  \def\s{\sigma} 
\def\t{\tau}

\def\cP{{\cal P}}

\def\pD{\partial D}

\def\wR{\widehat{R}}
\def\wS{\widehat{S}}
\newtheorem{theorem}{Theorem}
\newtheorem{lemma}[theorem]{Lemma}

\newtheorem{Remark}{Remark}
\newtheorem{claim}{Claim}

\def\cX{{\mathcal X}}

\def\area{\text{area}}
\newcommand{\wh}[1]{\widehat{#1}}

\newcommand{\rdup}[1]{{\left\lceil #1\right\rceil }}
\newcommand{\rdown}[1]{{\left\lfloor #1\right \rfloor}}

\newcommand{\brac}[1]{\left(#1\right)}

\newcommand{\bfrac}[2]{\left(\frac{#1}{#2}\right)}

\def\cE{{\mathcal E}}
\def\cL{\mathcal L}

\newcommand{\set}[1]{\left\{#1\right\}}

\def\E{\mathbb{E}}

\def\Pr{\mathbb{P}}

\newcommand{\ignore}[1]{}

\def\cA{{\mathcal A}}
\def\cB{{\mathcal B}}
\def\cC{{\mathcal C}}

\def\cX{{\mathcal X}}
\def\cU{{\mathcal U}}

\newcommand{\beq}[2]{\begin{equation}\label{#1}#2\end{equation}}
\newcommand{\mults}[1]{\begin{multline*}#1\end{multline*}}
\newcommand{\mult}[2]{\begin{multline}\label{#1}#2\end{multline}}

\usepackage{tikz}
\usetikzlibrary{decorations}
\usetikzlibrary{decorations.pathreplacing}
\usetikzlibrary{shapes.misc}
\usetikzlibrary{arrows}

\newcommand{\pic}[1]{
\begin{tikzpicture}
#1
\end{tikzpicture}}

\begin{document}
\author{Alan Frieze\thanks{Research supported in part by NSF grant DMS1952285 } and Wesley Pegden\thanks{Research supported in part by NSF grant DMS1363136 }\\Department of Mathematical Sciences\\Carnegie Mellon University\\Pittsburgh PA 15213}

\date{}
\title{Spanners in randomly weighted graphs: Euclidean case}
\maketitle
\begin{abstract}
Given a connected graph $G=(V,E)$ and a length function $\ell:E\to {\mathbb R}$ we let $d_{v,w}$ denote the shortest distance between vertex $v$ and vertex $w$. A $t$-spanner is a subset $E'\subseteq E$ such that if $d'_{v,w}$ denotes shortest distances in the subgraph $G'=(V,E')$ then $d'_{v,w}\leq t d_{v,w}$ for all $v,w\in V$. We study the size of spanners in the following scenario: we consider a random embedding $\cX_p$ of $G_{n,p}$ into the unit square with Euclidean edge lengths. For $\e>0$ constant, we prove the existence w.h.p. of $(1+\e)$-spanners for $\cX_p$ that have $O_\e(n)$ edges. These spanners can be constructed in $O_\e(n^2\log n)$ time. (We will use $O_\e$ to indicate that the hidden constant depends on $\e$.) There are constraints on $p$ preventing it going to zero too quickly.
\end{abstract}
\section{Introduction}
Given a connected graph $G=(V,E)$ and a length function $\ell:E\to \reals$ we let $d_{v,w}$ denote the shortest distance between vertex $v$ and vertex $w$. A $t$-spanner is a subset $E'\subseteq E$ such that if $d'_{v,w}$ denotes shortest distances in the subgraph $G'=(V,E')$ then $d'_{v,w}\leq t d_{v,w}$ for all $v,w\in V$. We say that the {\em stretch} of $E'$ is at most $t$. In general, the closer $t$ is to one, the larger we need $E'$ to be relative to $E$. Spanners have theoretical and practical applications in various network design problems. For a recent survey on this topic see Ahmed et al \cite{Aetal}. Work in this area has in the main been restricted to the analysis of the worst-case properties of spanners. In this note, we assume that edge lengths are random variables and do a probabilistic analysis.

We consider the case where $\ell_{i,j}=|X_i-X_j|$, where $\cX=\set{X_1,X_2,\ldots,X_n}$ are $n$ randomly chosen points from $[0,1]^2$. The case where the $n$ points are arbitrarily chosen is the subject of the book \cite{NS} by Narasimham and Smid. Section 15.1.2 of this book considers the random model where all $\binom{n}{2}$ edges between points are available. We denote this mode by $\cX_1$.  In this paper we consider a model where only a specified subgraph of the possible edges are available.  In particular,  we assume that edges exist between the points in $\cX$, independently with probability $p$. We denote this model by $\cX_p$. It constitutes a random embedding of the random graph $G_{n,p}$ into $[0,1]^2$. In the open problem session of CCCG 2009 \cite{CCCG}, O'Rourke asked the following question: for what values of $p$ is it true that w.h.p. $\cX_p$ is a $t$-spanner for $\cX_1$, where $t=O(1)$. Mehrabian and Wormald \cite{MW} showed that there is no choice of $p$ with this property. Frieze and Pegden \cite{FP} proved a related negative result and also considered the increase in the shortest path length when going from $\cX_1$ to $\cX_p$, 

Now $d_{i,j}=|X_i-X_j|$ when $\set{i,j}\in \cX_p$ implies that with probability one, a 1-spanner contains all $\approx\binom{n}{2}p$ edges. We prove the following: We write $O_{\e,\th}(\cdot)$ if the hidden constant in the big O notation depends on $\e,\th$. At the moment, in some places, these constants can grow rather fast, for example the dependence on $\e$ is only bounded by $\e^{-O(1/\e)}$.
\begin{theorem}\label{th2}
Suppose that the edges of $\cX_p$ are given their Euclidean length. Let $\e,\th>0$ be arbitrary fixed constants. We describe the construction of a $(1+\e)$-spanner $E_\e$ for $\cX_p$. 
\begin{enumerate}[(a)]
\item If  $np^{1+\th}\to\infty$ then $\E(|E_\e|)=O_{\e,\th}(p^{-\th}n)$.
\item If $\frac{1}{p\log 1/p}=o(\log^{1/2}n)$ then $|E_\e|\leq \E(|E_\e|)+O(n)$ w.h.p.
\end{enumerate}
\end{theorem}
The definition of $E_\e$ is given below in \eqref{Eeps}. On the other hand,
\begin{theorem}\label{th2a}
Suppose that the edges of $\cX_p$ are given their Euclidean length. Let $\e>0$ be an arbitrary fixed constant. If $np^2\to\infty$ then w.h.p. any $(1+\e)$-spanner for $\cX_p$ requires $\Omega(\e^{-1/2}n)$ edges.
\end{theorem}
\begin{Remark}
We stress that we describe a $(1+\e)$-spanner for $\cX_p$ and not for $\cX_1$. The results of  \cite{MW} and \cite{FP} rule out $O(1)$-spanners for $\cX_1$ that only use edges of $\cX_p$. This is because there will w.h.p. be pairs of points that are close together in Euclidean distance, but relatively far apart in $\cX_p$.
\end{Remark}
\begin{Remark}
We have assumed in Theorem \ref{th2} that $np^{1+\th}\to\infty$. If we were to allow $np^{1+\th}=o(1)$ then we would find that $np^{-\th}\gg n^2p$ and so the claimed size of our spanner is more than the likely number of edges in $\cX_p$. 
\end{Remark}
\begin{Remark}
The constant $\th$ is an artifact of our proof and we conjecture that it can be removed so that w.h.p. there is a $(1+\e)$-spanner of size $O_\e(n)$.
\end{Remark}
We note that when points are placed arbitrarily and {\em all pairs of points} are connected by an edge then the so-called $\Theta$-graph (defined below) produces a $(1+\e)$-spanner with $O(n/\e)$ edges. See Theorem 4.1.5 of \cite{NS}.

The argument we present for Theorem \ref{th2} can be easily adapted to deal with random geometric graphs $G_{\cX,r}$ for sufficiently large radius $r$. Here we generate $\cX$ as in Theorem \ref{th2} and now we join two vertices/points $X,Y$ by an edge if $|X-Y|\leq r$. See Penrose \cite{Pen} for an early book on this model. 
\begin{theorem}\label{th3}
If $r^2\gg \frac{\log n}{n}$ then w.h.p. there is a $(1+\e)$-spanner using $O(n\e^{-2})$ edges.
\end{theorem}
We note finally that Frieze and Pegden \cite{FP1} have also considered the case where edge lengths are independently exponential mean one. The results there are much tighter.

\section{Lower bound: the proof of Theorem \ref{th2a}}
It is quite easy to prove the lower bound in Theorem \ref{th2a}., so we begin with this. Given an edge $\set{A,B}\in E(\cX_p)$ we let $ellipse(A,B)$ be the ellipse with foci $A,B$ defined by $|X-A|+|X-B|\leq (1+\e)r$.  The edge $\set{A,B}$ is {\em lonely} if its length is $r$ and there is no $X\in \cX\cap ellipse(A,B)$ such that $\set{A,X},\set{B,X}$ are edges of $\cX_p$. Any $(1+\e)$-spanner must contain all of the lonely edges. Now $ellipse(A,B)$ has axes of size $a=(1+\e)r,b=(2\e+\e^2)^{1/2}r $ and so its volume is $\psi r^2$ where $\psi=\p(1+\e)(2\e+\e^2)^{1/2}/4$. By concentrating on points that are at least 0.1 from the boundary $\pD$ of $D=[0,1]^2$, we see that the expected number of lonely edges is at least 
\beq{alone}{
(0.64-o(1))\binom{n}{2}p\int_{r=0}^{0.8\sqrt{2}}\brac{1-\psi r^2p}^{n}\cdot2\p rdr\geq \frac{n^2\p}{2\psi}\int_{s=0}^{ \psi p}(1-s)^nds \geq  \frac{n\p}{3\psi},
}
where we have used $(1-p)^n=o(1)$.

Concentration around the mean follows will follow from the Chebyshev inequality. In preparation for this, observe that if $r\geq \r_\e=(20\log n/(np\psi))^{1/2}$ then $\brac{1-\psi pr^2}^{n}=o(n^{-10})$ and so going back to the first integral in \eqref{alone} we see that we can concentrate on lonely edges with $r\leq \r_\e$. Next consider the event $\cR$ that for each $A\in\cX$ there are at most $100\psi^{-1}\log n$ $\cX_p$ neighbors $B$ such that $|A-B|\leq \r_\e$. For a given $A$, the number of such close neighbors is distributed as a binomial with mean at most $20\p\psi^{-1}\log n$. So the Chernoff bounds imply that $\cR$ occurs with probability $1-o(n^{-10})$. So we let $Z$ denote the number of lonely edges $AB$ such that $|A-B|\leq \r_\e$ and observe that $\E(Z)=\Omega(n/\e^{1/2}p)$.

Observe also that given an edge $AB$ there are at most $O(\e^{-1}\log^2n)$ edges $CD$ for which $ellipse(A,B)\cap ellipse(C,D)\neq \emptyset$, assuming the occurrence of $\cR$. Write $AB\sim CD$ to denote a non-empty intersection of ellipses. Thus, if $\cL_{A,B}$ is the event that $AB$ is lonely, then
\mults{
\E(Z^2\mid\cR)\leq \sum_{AB}\sum_{CD\sim AB}\Pr(\cL_{A,B}\mid\cR)+\sum_{AB}\sum_{CD\not\sim AB}\Pr(\cL_{A,B},\cL_{C,D}\mid \cR)\\
\leq O(\E(Z)\e^{-1}\log^2n)+(1+o(1))\E(Z)^2=(1+o(1))\E(Z)^2.
}
The Chebyshev inequality implies that $Z$ is concentrated around its mean. This completes the proof of the lower bound in Theorem \ref{th2}.
\section{Upper bound: the proof of Theorem \ref{th2}}\label{th2p}
Suppose that $0<\e\ll 1$. It is perhaps instructive to consider the case where $p=1$ i.e. where $K_n$ is being embedded. In this case there are known, simple algorithms for finding a $(1+\e)$-spanner. For each $A\in\cX$ we define $\t$ cones $K_p(i,A),0\leq i<\t$ with apex $A$ and whose boundary rays make angles $i\e$ and $(i+1)\e$ with the horizontal. We then let $Y(i,A)$ denote the closest point in Euclidean distance to $A$ in $K_p(i,A)$ that is adjacent to $A$ in $\cX_p$. We put $Y(i,A)=\bot$ if there is no such $Y$ and let $d_{A,\bot}=\infty$. Also, define $i=i_{A,B}$ by $B\in K_p(i,A)$. When $p=1$, the Yao graph \cite{Yao} consists of the edges $(A,Y(i,A)), 0\leq i<\t,A\in\cX$. 
\begin{Remark}\label{rem0}
It is known that the path $P(A,B)=(Z_0=A,Z_1,\ldots,Z_m=B)$, where $Z_{i+1}=Y(i_{Z_i,B},Z_i)$ has length at most $(\cos\e-\sin\e)^{-1}|A-B|$ and so the Yao graph has stretch factor $1+\e+O(\e^2)$. 
\end{Remark}
When $p<1$, $P(A,B)$ may not exist in $\cX_p$ and we show below how to overcome this problem. 

We should also mention the very similar $\Theta$-graph \cite{NS0}. Here we replace $Y(i,A)$ by the point in $K(i,A)$ whose projection onto the bisector of $K(i,A)$ is closest to $A$. The $\Theta$-graph also has a stretch factor of at most $(\cos\e-\sin\e)^{-1}$.

Let
\beq{defR}{
r_\e=\bfrac{M_{\th,\e}}{np^{1+\th}}^{1/2}\text{ and }R_\e=\bfrac{K_\th\log n}{np^{1+\th}}^{1/2}.
}
where $M_{\th,\e}$ is sufficiently large to justify some inequalities claimed below.

Let
\[
E_1=\set{\set{A,B}\in \cX_p:|A-B|\leq r_{\e}}. 
\]
We have 
\beq{E00}{
\E(|E_1|)\leq \binom{n}{2}\p r_{\e}^2p\leq  \frac{M_{\th,\e} n}{2p^{\th}}
} 
and then we can assert that
\beq{E0}{
|E_1|\leq \frac{M_{\th,\e} n}{p^{\th}}\ w.h.p.
}
using the Chebyshev inequality. Here we can use the fact that the events of the form $\set{|A-B|\leq r_\e}$ are pair-wise independent.

Let
\beq{E2}{
E_2=\set{(A,Y(i,A)):A\in\cX,i\in\set{0,1,\ldots,\t-1}}\text{ so that }|E_2|=O(n/\e).
}
The next two lemmas will discuss the case where $A,B$ are sufficiently distant.
\begin{lemma}\label{far}
If $|A-B|\geq R_\e$ then with probability $1-o(n^{-10})$, $|A-Y|\leq\e|A-B|$, where $Y=Y(i_{A,B},A)$.
\end{lemma}
\begin{proof}
We have
\[
\Pr(|A-Y|>\e|A-B|)\leq (1-\e\p (\e R_\e)^2p/2)^{n-1}\leq n^{-\e^3\p M_{\th,\e}/3p^{\th}}.
\]
The 2 in the middle expression allows half the cone to be outside $[0,1]^2$.
\end{proof}
\begin{lemma}\label{farx}
If $r\geq R_\e$ then  with probability $1-o(n^{-10})$, $d_{A,B}\leq (1+4\e)|A-B|$.
\end{lemma}
\begin{proof}
Let $X_1,X_2$ be points on the line segment $AB$ at distance $|A-B|/3,2|A-B|/3$ from $A$ respectively. Let $B_i,i=1,2$ be the ball of radius $\e r$ centred at $X_i$. Let $A_1$ be the set of $\cX_p$ neighbors of $A$ in $X_1$ and let $A_2$ be the set of $\cX_p$ neighbors of $B$ in $X_2$. $\cE_i,i=1,2$ be the event that $|A_i|\geq \p r^2np/10$. Then the Chernoff bounds imply that 
\[
\Pr(\cE_1\wedge\cE_2)\geq 1-2e^{-\p r^2np/1000}=1-O(n^{-\p M_{\th,\e}/1000p^{\th}}).
\]
Let $\cE_3$ be the event that there is an $\cX_p$ edge between $A_1$ and $A_2$. Then
\[
\Pr(\cE_3\mid \cE_1\wedge\cE_2)\geq 1-\brac{1-p}^{r^4n^2p^2/100}=1-O(n^{-K^2_{\th,\e}/100p^{\th}}).
\]
Finally note that if $\cE_i,i=1,2,3$ all occur then $d_{A,B}\leq (1+4\e)|A-B|$. (4 is trivial and avoids any computation.)
\end{proof}
For $A,B\in\cA$ we let $P_{A,B}$ denote the shortest path between $A,B$ in $\cX_p$ and we let $d_{A,B}$ denote the length of $P_{A,B}$. 

Let 
\beq{Ralf}{
\cB_\e=\set{(A,B):\;d_{A,B}\geq (1+\e)|B-A|\text{ and }r=|A-B|\geq r_\e}
}
and
\[
E_3=\bigcup_{(A,B)\in \cB_{\e}}E(P_{A,B}).
\] 

Let 
\[
\cC_{\e}=\set{(A,B):\;d_{A,B}\leq (1+\e)|B-A|\text{ and }r=|A-B|\in[ r_\e,R_\e]\text{ and }|A-Y|\geq \e|A-B|},
\]
where $Y=Y(i_{A,B},A)$. Let 
\[
E_4=\bigcup_{(A,B)\in \cC_{\e}}E(P_{A,B}).
\] 
We show in Lemmas \ref{longpaths} and \ref{E3} that the expected sizes of the sets $E_3,E_4$ are $O_\e(n)$. Let 
\beq{Eeps}{
E_\e=\bigcup_{i=1}^4E_i.
}
\paragraph{Time:} The construction of $E_\e$ can obviously be done in polynomial time. The most time consuming parts being solving the all pairs shortest path problems defined by $E_3,E_4$. We show below that these sets consist of $O_\e(n)$ edges in expectation. So the expected time to solve these $O(n)$ single source problems via Dijkstra's algorithm is $O_\e(n^2\log n)$, see Fredman and Tarjan \cite{FT}.

For $X,Y\in\cX$ we let $\hd_{X,Y}$ denote the length of the path from $X$ to $Y$ constructed by the following procedure: Given $A,B\in\cX$ where $\set{A,B}\notin E$ we construct a path $A=Z_0,Z_1,\ldots,Z_k=B$ as follows: in the following, $Y_j=Y(i,Z_j)$ for $B\in K(i,Z_j),j\geq 0$.

{\sc Construct:}
\begin{enumerate}[D1]
\item If $\set{Z_j,B}\in E_1$ then  use $P_{Z_j,B}$ to complete the path, otherwise,
\item If $|Z_j-Y_j|>\e|Z_j-B|$ then  use $P_{Z_j,B}$ to complete the path, otherwise,
\item If $d_{Y_j,B}\geq (1+5\e)|Y_j-B|$ then use $P_{Z_j,B}$ to complete the path, otherwise 
\item $Z_{j+1}\gets Y_j$.
\end{enumerate}
\begin{Remark}\label{rem+}
We observe that Lemma \ref{far} implies that with probability $1-o(n^{-10})$ we do not use $P_{Z_j,B}$ for $|Z_j-B|\geq R_\e$. Denote the corresponding event by $\cU$.
\end{Remark}
The next lemma is used to estimate the quality of the path built by {\sc construct}. (We can obviously replace $8\e$ by $\e$ in order to get a $(1+\e)$-spanner.)
\begin{lemma}\label{sp}
{\sc construct} produces a path of length at most $(1+7\e)d_{A,B}$.
\end{lemma}
\begin{proof}
Let $A=Z_0,Z_1,\ldots,Z_k=B$ be the sequence defined by {\sc construct}. If $k=1$ then {\sc construct} uses that path $P_{A,B}$ which has stretch one. Otherwise, let $d_j=|Z_j-B|$ for $0\leq j\leq k$ and observe that it is a monotone decreasing sequence.
Define $\bar Z_{j+1}$ to the point on the segment $Z_jZ_k$ such that $|\bar Z_{j+1}-Z_k|=|Z_{j+1}-Z_k|$.  The assumption that $|Z_j-Z_{j+1}|\leq \e|Z_j-Z_k|$ implies that $\angle Z_{j+1}Z_k\bar Z_{j+1}<\pi/2$, and thus that the ratio   
\beq{Zd}{
\frac{|Z_{j+1}-Z_j|}{d_j-d_{j+1}}
}
can be bounded by considering the case where $\angle Z_{j+1}Z_k\bar Z_{j+1}=\pi/2,$ as it is drawn in Figure \ref{f.right}.

We have in that case that $\sin\e=\frac{d_{j+1}}{|Z_j-Z_{j+1}|}$ and $\cos\e=\frac{d_{j}}{|Z_j-Z_{j+1}|}$, giving $d_j-d_{j+1}=(\cos\e-\sin\e)|Z_j-Z_{j+1}|$. So, if {\sc construct} only uses D4 then the length $L_{A,B}$ of the path constructed satisfies
\[
L_{A,B}=\sum_{j=0}^{k-1}|Z_{j+1}-Z_j|\leq (\cos\e-\sin\e)\sum_{j=1}^k(d_j-d_{j+1})=(\cos\e-\sin\e)|A-B|\leq (\cos\e-\sin\e)d_{A,B}.
\]
Suppose that {\sc construct} uses a path in D1,D2 or D3. If $k=1$ then {\sc construct} uses a shortest path from $A$ to $B$ in $\cX_p$. Assume then that $k\geq 2$. It follows from the above argument that 
\[
\sum_{j=0}^{k-2}|Z_{j+1}-Z_j|\leq (\cos\e-\sin\e)||A-Z_{k-1}|.
\] 
Now,
\[
d_{Z_{k-1},B}\leq |Z_{k-2}-Z_{k-1}|+d_{Z_{k-2},B}\leq \e|Z_{k-2}-B|+(1+5\e)|Z_{k-2}-B|
\]
So,
\begin{align*}
L_{A,B}&\leq(\cos\e-\sin\e)||A-Z_{k-1}|+(1+6\e)|Z_{k-2}-B|\\
&\leq (1+6\e)(|A-Z_{k-2}|+|Z_{k-2}-B|)\\
&\leq (1+6\e)(\cos\e-\sin\e)|A-B|.
\end{align*}
\end{proof}
\begin{figure}[h]
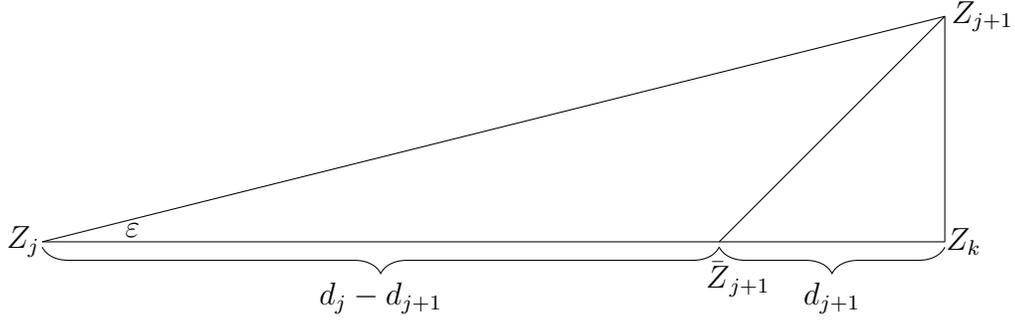

\begin{center}
\pic{
\draw (0,0) -- (12,0);
\draw (0,0) -- (12,3);
\draw (12,0) -- (12,3);
\draw (9,0) -- (12,3);

\draw [decorate,decoration={brace,amplitude=10pt,mirror},yshift=-2pt]
(0,0) -- (9,0) ;

\draw [decorate,decoration={brace,amplitude=10pt,mirror},yshift=-2pt]
(9,0) -- (12,0) ;

\node at (-1/4,0) {$Z_j$};
\node at (12.5,3) {$Z_{j+1}$};
\node at (12.25,0) {$Z_k$};

\node at (9.25,-.5) {$\bar Z_{j+1}$};

\node at (1.2,.15) {$\e$};

\node at (4.5,-.75) {$d_j-d_{j+1}$};
\node at (10.5,-.75) {$d_{j+1}$};
}
\end{center}
\caption{Extreme case for \eqref{Zd}}\label{f.right}
\end{figure}
We argue next that 
\begin{lemma}\label{paths}
The edges of the paths $P_{Z_j,B}$ used in {\sc construct} are contained in $E_1\cup E_3\cup E_4$. Furthermore, only edges of length at most $R_\e$ contribute to $E_3,E_{4}$.
\end{lemma}
\begin{proof}
First consider the path $P=P_{Z_j,B}$ used in D1. Because $\set{Z_j,B}\in E_1$, we have that $d_{Z_j,B}\leq r_\e$ and so all the edges of $P_{Z_j,B}$ are also in $E_1$. 

Next consider the path $P=P_{Z_j,B}$ used in D2. If $d_{Z_j,B}\geq (1+\e)|Z_j-B|$ then $E(P)\subseteq E_3$. Otherwise, $E(P)\subseteq E_4$.

Now consider the path $P=P_{Z_j,B}$ used in D3. If $d_{Z_j,B}\geq (1+\e)|Z_j-B|$ then $E(P)\subseteq E_3$. So assume that $d_{Z_j,B}\leq (1+\e)|Z_j-B|$. If $|Z_j-Y_j|\geq \e|Z_j-B|$ then $E(P)\subseteq E_4$. So assume that $|Z_j-Y_j|\leq \e|Z_j-B|$. At this point we have
\[
\brac{1+5\e}|Y_j-B|\leq d_{Y_j,B}\leq |Z_j-Y_j|+d_{Z_j,B}\leq (1+2\e)|Z_j-B|\leq \brac{1+2\e}(|Z_j-Y_j|+|Y_j-B|).
\]
This implies that $|Z_j-Y_j|\geq 3\e|Y_j-B|/(1+2\e)$. If $|Y_j-B|\geq |Z_j-B|/2$ then we have $E(P)\subseteq E_4$. So assume that $|Y_j-B|< |Z_j-B|/2$. But then $|Z_j-Y_j|\geq |Z_j-B|-|Y_j-B|\geq |Z_j-B|/2$, a contradiction.
\end{proof}
The next two lemmas bound the expected number of edges in the sets $E_3,E_{4}$. 
\subsection{$\E(|E_3|)$}\label{secE3}
\begin{lemma}\label{longpaths}
$\E(|E_3|)=O_{\th,\e}\bfrac{n}{p^{\th}}$.
\end{lemma}
\begin{proof}
  Fix a pair of points $A,B\in \cX$ and let $r=|A-B|$ where $r_\e\leq r\leq R_\e$ ($r_\e,R_\e$ defined in \eqref{Ralf}).  Note next that shortest paths are always induced paths. We let $\cL_{K,k,A,B}$ denote the set of induced paths from $A$ to $B$ with $k+1\geq 2$ edges in $\cX_p$, of total length in $[(1+K\e)r,(1+(K+1)\e)r]$.

We let $L_{K,k,A,B}=|\cL_{K,k,A,B}|$. Then we have
\begin{equation}\label{Easum}
|E_3|\leq\sum_{A,B\in\cX}\sum_{k,K=1}^\infty k|\set{P\in\cL_{K,k,A,B}}|.
\end{equation}
This is because if $d_{A,B}\geq (1+\e)|A-B|$ then the shortest path from $A$ to $B$ has its length in $J_{K,r}=[(1+K\e)r,(1+(K+1)\e)r]$, for some $K\geq 1$. Next define, for $L\geq 1$,
\[
F(L,\e):=(2L\e+L^2\e^2)^{1/2}.
\]
\begin{claim}\label{cl1a}
There are constants $\La,c$ such that for $K\geq 1$,
\beq{sumBa}{
  \E\brac{L_{K,k,A,B}\middle| |A-B|=r} \leq  \bfrac{\La F(K+1,\e)(1+(K+1)\e) r^2np(1-p)^{(k-1)/2}}{k^2(K\e(1+K\e))^{1/4}}^{k}e^{-cF(K+1,\e)(1+(K+1)\e)r^2np}.
}
\end{claim}
{\bf Proof of Claim \ref{cl1a}}:
Let $E_{A,B}(L)$ denote the ellipse with centre the midpoint of $AB$, foci at $A,B$ so that one axis is along the line through $AB$ and the other is orthogonal to it. The axis lengths $a,b$ being given by $a=(1+L\e)r$ and $b=r((1+L\e)^2-1)^{1/2}=rF(L,\e)$. Thus $E_{A,B}(L)$ is the set of points whose sum of distances to $A,B$ is at most $(1+L\e)r$. 

Given $k$ points $P_1,\dots,P_k$, the path  $P=(A=P_0,P_1,\dots,P_k,P_{k+1}=B)$ is of length at most $(1+(K+1)\e)r$ only if all these points lie in $E_{A,B}(K+1)$. Thus for all $i$ the point $P_{i+1}$ lies in an ellipse with axes $2a,2b$ centred at $P_i$.  Here we are using the fact that if a point $x$ lies in an ellipse $E$ then $E$ is contained in a copy of $2E$ centered at $x$. Indeed, suppose that $(x_i,y_i),i=1,2$ are two points in the ellipse $E=\set{\frac{x^2}{\xi^2}+\frac{y^2}{\eta^2}\leq 1}$.
Then
\begin{equation}\label{e.2E}
\frac{(x_1-x_2)^2}{\xi^2}+\frac{(y_1-y_2)^2}{\eta^2}\leq \frac{2(x_1^2+x_2^2)}{\xi^2}+ \frac{2(y_1^2+y_2^2)}{\eta^2}=2\sum_{i=1}^2\brac{\frac{x_i^2}{\xi^2}+\frac{y_i^2}{\eta^2}}\leq 4.
\end{equation}
It follows that $(x_1,y_1)$ is contained in a copy of $2E$ centered at $(x_2,y_2)$.

So, the probability of the event that $(A=P_0,P_1,\dots,P_k)$ is in $E_{A,B}(K+1)$ is at most $\prod_{i=1}^k\Pr(\cP_i)$ where $\cP_i$ is the event that $P_{i+1}$ is in the ellipse congruent to $2E_{A,B}(K+1)$, centred at $P_i$. So,
\beq{cPi}{
\Pr(\text{$(A=P_0,P_1,\dots,P_k,B)$ is in $E_{A,B}(K+1)$})\leq (\p r^2F(K+1,\e)(1+(K+1)\e))^kp.
}
The final $p$ factor is $\Pr(\set{P_k,B}\in E)$. Given $\cP_1,\cP_2,\ldots,\cP_k$ the length of $P$ is at most the sum $Z_1+\cdots+Z_k$ of independent random variables where $Z_i$ is the distance to the origin of a random point in an ellipse with axes $2a,2b$ centred at the origin.
\begin{lemma}\label{tomasz}
\begin{enumerate}[(a)]
\item $Z_1$ is distributed as $2(U(a^2\cos^2(2\p V)+b^2\sin^2(2\p V)))^{1/2}$ where $U,V$ are independent uniform $[0,1]$ random variables.
\item $Z_1$ stochastically dominates $\z^{-1/2} U^{1/2}(K\e(1+K\e))^{1/4}r$ for some $\z>0$.
\end{enumerate}
\end{lemma}
\begin{proof}
(a) This follows from the fact that a point in $E$ is of the form $(a\cos2\p\th,b\sin2\p\th)u$ where $0\leq u,\th\leq 1$.

(b) We have 
\begin{align*}
\Pr(Z_1\leq x)&=\Pr\brac{U\leq \frac{x^2}{4(a^2\cos^2(2\p V)+b^2\sin^2(2\p V))}}\\ &=\E\brac{\min\set{1,\tfrac14x^2(a^2\cos^2(2\p V)+b^2\sin^2(2\p V))^{-1}}}\\
&\leq \min\set{1,\E\brac{\frac{x^2}{a^2\cos^2(2\p V)+b^2\sin^2(2\p V)}}}.
\end{align*}
Now
\begin{align*}
\E\brac{\frac{1}{a^2\cos^2(2\p V)+b^2\sin^2(2\p V)}}&=\frac{2}{\p}\int_{z=0}^{\p/2}\frac{dz}{a^2\cos^2(z)+b^2\sin^2(z)}=\frac{2}{\p}\int_{z=0}^{\p/2}\frac{dz}{a^2\sin^2(z)+b^2\cos^2(z)}\\
&=\frac{2}{\p}\int_{z=0}^{\p/2}\frac{dz}{(a^2-b^2)\sin^2(z)+b^2}\\
&\leq\frac{4}{\p}\int_{z=0}^{1/2}\frac{dz}{(a^2-b^2)z^2+b^2}+O\bfrac{1}{a^2}\\
&=\frac{4}{\p r^2}\int_{z=0}^{1/2}\frac{dz}{z^2+2K\e+K^2\e^2}+O\bfrac{1}{(1+(K+1)\e)^2r^2}.\\
&=\frac{4}{\p r^2}\frac{\arctan{\bfrac{1}{2(2K\e+K^2\e^2)^{1/2}}}}{(2K\e+K^2\e^2)^{1/2}}++O\bfrac{1}{(1+(K+1)\e)^2r^2}.
\end{align*}
So
\[
\Pr(Z_1\leq x)\leq \frac{\z x^2}{(K\e(1+K\e))^{1/2} r^2}
\]
for some $\z>0$.

This implies that $Z_1$ dominates $\z^{-1/2} U^{1/2}(K\e(1+K\e))^{1/4}r$. 
\end{proof}
Lemma 9 of Frieze and Tkocz \cite{FT1} implies that if $U_1,U_2,\ldots,U_k$ are independent copies of $U^{1/2}$ then
\[
\Pr(U_1^{1/2}+U_2^{1/2}+\cdots+U_k^{1/2}\leq u)\leq \frac{(2u)^{2k}}{(2k)!}.
\] 
Putting $u=\frac{\a^{1/2}}{ (K\e(1+K\e))^{1/4}r}$, we see that 
\beq{disk}{
\Pr(Z_1+Z_2+\cdots+Z_k\leq (1+(K+1)\e)r)\leq \bfrac{\a(1+(K+1)\e)}{(K\e(1+K\e))^{1/4}}^{k} \frac{2^k}{(2k)!}\leq  \bfrac{\a(1+(K+1)\e)}{(K\e(1+K\e))^{1/4}}^{k} \frac{e^{2k}}{k^{2k}2^k}.
}
Thus, given $k$ random points $P_1,\dots,P_k$, the probability that $A,P_1,\dots,P_k$ is an induced path of length $\leq (1+(K+1)\e)r$ is at most
\[
\bfrac{\La F(K+1,\e)(1+(K+1)\e)r^2np(1-p)^{(k-1)/2}}{k^2(K\e(1+K\e))^{1/4}}^{k}.
\]
To get the exponential term in  \eqref{sumBa}, we need to also make make use of the fact that $d_{A,B}\geq (1+\e K)r$.  

{\bf Case 1: $K\e\leq 1$:} Let $\g=\rdup{1+\th^{-1}}$. We define $\g$ rhombi, $R_i,i=1,2,\ldots,\g$. We partition $AB$ into $\g$ segments $L_1,L_2,\ldots,L_\g$ of length $r/\g$. The rhombus $R_i$ has one diagonal $L_i$ and another diagonal of length $h=((K+1)\e)^{1/2}r/10\g$ that is orthogonal to $AB$ and bisects it. Finally let $\wR_i=R_i\cap [0,1]^2$. Note that $\wR_i$ has area at least 1/2 of the area of $R_i$. Thus if $K\geq 1$ then since $K\e\leq 1$,
\beq{area}{
\a\geq \a_i=\area(\wR_i)\geq\frac{((K+1)\e)^{1/2}r^2}{20\g}\geq \frac{\a}{100} 
}
where 
\[
\a=\frac{F(K+1,\e)(1+(K+1)\e)r^2}{\g}.
\]

For a pair of points $A,B$ and set $X\subseteq \cX$, let $\td_{A,B}(X)$ denote the minimum length of a path\\ $Q=(A,S_1,S_2,\ldots,S_\g,B)$ in $\cX_p$ where $S_i\in \wR_{i}\setminus X$. Here $X$ will stand for $P_1,P_2,\ldots,P_k$ in the analysis below. Furthermore we can restrict our attention to $|X|=k=o(n)$, as shown in \eqref{k1} below. We first wish to show that
\beq{ellQ}{
\ell(Q)<(1+K\e)r\text{ for all choices of }S_1,S_2,\ldots,S_\g.
}
 Now fix $i$ and consider the function $f(S)=\ell(A,S_1,S_2,\ldots,S_{i-1},S,S_{i+1},\ldots,,S_\g,B)$. This is a convex function of $S$ and so it is maximised at an extreme point of $\wR_{i}\setminus X$. Thus to verify \eqref{ellQ}, it is enough to check paths that only use the vertices of the rhombi. We claim that
\beq{lQ}{
\ell(Q)\leq \g\brac{4h^2+\frac{1}{\g^2}}^{1/2}r\leq r\g\brac{2h+\frac{1}{\g}}\leq (1+(K+1)\e)r
}
where we have used $K\e\leq 1$ for the last inequality. Equation \eqref{lQ} follows from the fact that $\brac{4h^2+\frac{1}{\g^2}}^{1/2}r$ maximises the distance between points in adjacent rhombi.
 
Let $Z$ denote the number of paths $Q$ such that all edges exist in $\cX_p$. We use Janson's inequality \cite{JAN} to bound the probability that $Z=0$. We have, with $\n=n-|X|=n-o(n)$,
\[
\E(Z)= \n(\n-1)\cdots(\n-\g+1) p^{\g+1}\prod_{i=1}^\g \a_i\geq \bfrac{\a np}{100}^\g \frac{p}2.
\]
Then for a pair of paths $Q,Q'$ let $\r(Q,Q'),\s(Q,Q')$, denote the number of vertices and edges the $Q,Q'$ have in common. (Exclude $A,B$ from this count.) We write $Q\sim Q'$ to mean that $\r(Q,Q')>0$. Then, 
\beq{D2}{
\bar \D=\sum_{Q\sim Q'}\Pr(Q,Q')\leq 2^{2\g}\sum_{\substack{1\leq \s\leq \g+1\\\s\leq \r\leq 2\s}}(\a n)^{2\g-\r} p^{2\g+2-\s}\leq 2^{2\g+1}(\a n)^{2\g-1}p^{2\g+1}.
}
{\bf Explanation for \eqref{D2}} Because $r\geq r_ \e$, we have $\a np\gg1$. Thus the sum in \eqref{D2} is dominated by the term $\r=\s=1$ where $Q,Q'$ only share an edge incident to $A$ or $B$. The factor $2^{2\g}$ accounts for the places on $Q,Q'$ that share a common vertex.

It follows that if $K\geq 1$ then
\mults{
\r_{k,K,\e}=\Pr(d^*_{A,B}\geq (1+K\e)r\mid |A-B|=r,P_1,\dots,P_k)\leq \exp\set{-\frac{\E(Z)^2}{2\bar \D}}\leq  \\ \exp\set{-\frac{F(K+1,\e)(1+(K+1)\e)r^2np}{2^{2\g+4}10^{4\g}\g}}\leq\exp\set{-\frac{M_{\th,\e}F(K+1,\e)(1+(K+1)\e)}{2^{2\g+4}10^{4\g}\g p^{\th}}}
}

{\bf Case 2: $K\e\geq 1$:} Let $R$ be the rectangle with center the midpoint of $AB$ and one side of length $(1+(K+1)\e/10)r$ parallel to $AB$ and the other of side $K\e/10$ orthogonal to $AB$. We partition $R$ into rectangles $W_1,W_2,\ldots,W_\g$ where each $W_i$ has side lengths $(1+(K+1)\e/10)r/\g$ and $K\e/10$. Putting $\wh W_i=W_i\cap [0,1]^2,i=1,2,\ldots,\g$ we see that all we need do now is to prove the equivalent of \eqref{area} and \eqref{ellQ}. Then,
\[
\area(\wh W_i)\geq \brac{1+\frac{(K+1)\e}{10}}\frac{K\e}{20\g}r^2\geq \frac{F(K+1,\e)(1+(K+1)\e)}{1000\g}r^2.
\]
We have used $K\e\geq 1$ to justify the second inequality. 

We further have that for all $S_i\in\wS_i,i=1,2,\ldots,\g$ that, using the triangle inequality, 
\[
\ell(A,S_1,\ldots,S_\g,B)\leq \g\brac{1+\frac{(K+1)\e}{10}}\frac{r}\g+\g\brac{\frac{K\e}{10}+\frac{4(K+1)\e}{10}}\frac{r}\g<(1+(K+1)\e)r.
\]

Thus, the probability $\rho_{k,K,\e}$ defined above satisfies
\[
\rho_{k,K,\e}\leq  \bfrac{\La F(K+1,\e)(1+(K+1)\e) r^2np(1-p)^{(k-1)/2}}{k^2}^{k} e^{-cF(K+1,\e)(1+(K+1)\e)r^2np},
  \]
and the claim follows by linearity of expectation.

{\bf End of proof of Claim \ref{cl1a}}

It will be convenient to replace $r$ by $\frac{\r}{(np)^{1/2}}$ and write $J_\r=[\frac{\r}{n^{1/2}},\frac{\r+1}{n^{1/2}}]$ and let $\r_{\min}=r_\e (np)^{1/2}$. Then, 
\begin{align}
&\E(|E_3|)\nonumber\\
&\leq \binom{n}{2}\sum_{\r=\r_{\min}}^{\infty}\sum_{K=1}^{\infty}\sum_{k=1}^{n-2}k \bfrac{\La F(K+1,\e)(1+(K+1)\e) r^2np(1-p)^{(k-1)/2}}{k^2(K\e(1+K\e))^{1/4}}^{k}\nn\\
&\hspace{3in}\times e^{-cF(K+1,\e)(1+(K+1)\e)r^2np}\Pr(|A-B|\in J_\r)\nonumber\\
&\leq \binom{n}{2}\p \sum_{\r=\r_{\min}}^{\infty}\sum_{K=1}^{\infty}\sum_{k=1}^{n-2}k \bfrac{\La F(K+1,\e)1+(K+1)\e) r^2np(1-p)^{(k-1)/2}}{k^2(K\e(1+K\e))^{1/4}}^{k}\nn\\ &\hspace{3in}\times e^{-cF(K+1,\e)(1+(K+1)\e)r^2np}\bfrac{2\r+1}{n}\nonumber\\
&\leq 2\p n\sum_{k=1}^{n-2}k\sum_{K=1}^{\infty} \bfrac{\La  F(K+1,\e)1+(K+1)\e)(1-p)^{(k-1)/2}}{k^2(K\e(1+K\e))^{1/4}}^{k} \sum_{\r=\r_{\min}}^{\infty} e^{-cF(K+1,\e)(1+(K+1)\e)\r^2}\r^{2k+1}\nonumber\\
&\leq 2\p n\sum_{k=1}^{n-2}k\sum_{K=1}^{\infty}\bfrac{\La  F(K+1,\e)1+(K+1)\e)(1-p)^{(k-1)/2}}{k^2(K\e(1+K\e))^{1/4}}^{k} \int_{s=0}^{\infty} e^{-cF(K+1,\e)(1+(K+1)\e)s}s^{k}ds\nonumber\\
&=2\p n\sum_{k=1}^{n-2}k\sum_{K=1}^{\infty} \bfrac{\La  F(K+1,\e)1+(K+1)\e)(1-p)^{(k-1)/2}}{k^2(K\e(1+K\e))^{1/4}}^{k} \bfrac{1}{cF(K+1,\e)(1+(K+1)\e)}^{k+1}k!\nonumber\\
&\leq 2\p n\sum_{k=1}^{n-2}k\bfrac{\La(1-p)^{(k-1)/2}}{k\e^{1/4}}^{k}\sum_{K=1}^{\infty}  \bfrac{1}{cF(K+1,\e)(1+(K+1)\e)} \frac{1}{(K(1+K\e))^{k/4}}\label{E2aa}\\
&=O_\e(n).\nn
\end{align}

\end{proof}

\subsection{$\E(|E_{4}|)$}\label{secE4}
\begin{lemma}\label{expectedpathsab}
The expected number of $(k+1)$-edge induced paths of length at most $(1+\e)r$ from $A$ to $B$ in $\cX_p$ can be bounded by
\begin{equation}
  \brac{n\p r^2p(1-p)^{(k-1)/2}\frac{\e(1+\e)^3e^2}{2k^2}}^k(1-\p \e^3 r^2p)^{n-k-2}p.
\end{equation}
\end{lemma}
\begin{proof}
  Let $\rho_k$ denote the probability that $k$ fixed points $X_1,\dots,X_k$ satisfy that:
  \begin{itemize}
  \item $A=X_0,X_1,\dots,X_k$ is an induced path
  \item For all $i=1,\dots,k$, $X_i$ lies in a copy of the ellipse $2\cdot E_{A,B}$, translated to be centered at $X_{i-1}$, and
  \item The total length of the path has total length at most $(1+\e)r$.
\item $\set{X_k,B}\in \cX_p$.
  \end{itemize}
From the discussion immediately prior to \eqref{e.2E}, we see that $\r_k$ bounds the probability that the the path has total length at most $(1+\e)r$.  So we have that
  \[
  \rho_k\leq (2\p\e(1+\e)r^2p)^k(1-p)^{k(k-1)/2}\bfrac{e^2(1+\e)^2}{2k^2}^kp.
  \]
  Thus, by linearity of expectation, the number of induced paths $A=X_0,\dots,X_k$ such that
  \begin{itemize}
  \item the total length of the path is at most $(1+\e)r$, and
  \item no point off the path lies within distance $\e r$ of $A$ in the cone $K(i,A)$
  \end{itemize}
  is at most
  \mults{
n^k(2\p\e(1+\e)r^2p)^k(1-p)^{k(k-1)/2}\bfrac{e^2(1+\e)^2}{2k^2}^k(1-\p \e^3 r^2p)^{n-k-2}p=\\ \brac{\frac{n\p r^2p(1-p)^{(k-1)/2}}{1-\p \e^3 r^2p}\frac{\e(1+\e)^3e^2}{2k^2}}^k(1-\p \e^3 r^2p)^{n-k-2}p\leq\\
 \brac{n\p r^2p(1-p)^{(k-1)/2}\frac{\e(1+\e)^3e^2}{3k^2}}^k(1-\p \e^3 r^2p)^{n-k-2}p.
}
\end{proof}

\begin{lemma}\label{E3}
$\E(|E_4|)=O_\e(n)$.
\end{lemma}
\begin{proof}

We have 
\begin{align}
\E(|E_{4}|)&\leq2\p\int_{r=r_\e}^{R_\e}\binom{n}{2}p\sum_{k=1}^\infty k\brac{n\p r^2p(1-p)^{(k-1)/2}\frac{\e(1+\e)^3e^2}{3k^2}}^k(1-\p\e^3 r^2p)^{n-k-2} rdr\label{E2bb}\\
&\leq 2\p\binom{n}{2}p\sum_{k=1}^\infty k\int_{r=r_\e}^{R_\e}\bfrac{e\p\e r^2np(1-p)^{(k-1)/2}}{k^2}^ke^{-\p \e^3r^2np} rdr\nn\\
&\leq \frac{n}{\e^3 }\sum_{k=1}^\infty k\int_{s=A}^\infty\bfrac{e\e(1-p)^{(k-1)/2}  s}{\e^3 k^2}^ke^{-s}ds,\label{pto0}
\end{align}
where $A=\p\e^2 r_\e^2np=M_{\th,\e}p^{-\th}$. Now,
\beq{pconst}{
I_k=\int_{s=A}^\infty s^ke^{-s}=k!\sum_{\ell=0}^k\frac{e^{-A}A^\ell}{\ell!}\leq 2e^{-A}A^k, \qquad\text{if }k\leq A/2.
}
(Use $I_k=kA^{k-1}e^{-A}+kI_{k-1}$ to obtain the equation.)

Using \eqref{pconst} in \eqref{pto0} we get, for small $\e$ and $k_0=10\log_b1/\e$ where $b=1/(1-p)$,
\mult{qaz1}{
\sum_{k=1}^{k_0} k\int_{s=A}^\infty \bfrac{e(1-p)^{(k-1)/2}  s}{\e^2 k^2}^ke^{-s}ds \leq e^{-A}\sum_{k=1}^{k_0}\bfrac{eA }{\e^2 k^2}^k \leq \\ Ak_0\exp\set{-M_{\th,\e}p^{-\th} +(M_{\th,\e}p^{-\th})^{1/2}}
\leq \exp\set{-\frac{M_{\th,\e}}{2p^{\th}}},
}
where we have used $(eC/x^2)^x\leq e^{2C^{1/2}}$ for $C>0$.

Finally,
\beq{qaz2}{
\sum_{k=k_0+1}^{\infty} k\int_{s=A}^\infty \bfrac{e(1-p)^{(k-1)/2}  s}{\e^2 k^2}^ke^{-s}ds \leq  \int_{s=A}^\infty e^{-s}\sum_{k=k_0+1}^{\infty}\bfrac{2e\e^3  s}{k^2}^kds \leq \int_{s=A}^\infty e^{-(1-\e)s}ds \leq e^{-A/2}.
}
Substituting \eqref{qaz1}, \eqref{qaz2} into \eqref{pto0} we see that $\E(|E_{4}|)=O\bfrac{n}{\e^3}$.
\end{proof}
We have argued that {\sc construct} builds a $(1+\e)$-spanner w.h.p. The set of edges in this spanner is that of $\bigcup_{i=0}^4E_i$. Part (a) of Theorem \ref{th2} now follows from \eqref{E00}, \eqref{E2}, Lemma \ref{longpaths} and Lemma \ref{E3}.
\subsection{Concentration of measure}\label{com}
Theorem \ref{th2} claims a high probability result. We apply McDiarmid's inequality \cite{mcd} to prove that $|E_3|,|E_4|$ are within range w.h.p. We do not seem to be able to apply the inequality directly and so a little preparation is necessary. We first let $m=\rdown{1/R_\e}$ and divide $[0,1]^2$ into a grid of $m^2$ subsquares $\cC=(C_1,C_2,\ldots,C_{m^2})$ of size $1/m\geq R_\e$. The Chernoff bounds imply that with probability $1-o(n^{-10})$ each $C\in \cC$ contains at most $\r_0=2nR_\e^2$ randomly chosen points of $\cX$. Suppose that we generate the points one by one and color a point blue if it is one of the first $\r_0$ points in its subsquare. Otherwise, color it red. Let $\cB$ be the event that  all points of $\cX$ are blue and we note that 
\beq{cB}{
\Pr(\cB)=1-o(n^{-10}).
}
Let
\beq{k1}{
\k_1=\frac{100\log^{1/2}n}{p}.
}
The significance of $\k_1$ is that the factors $(1-p)^{k(k-1)/2}$ in equations \eqref{E2aa} and \eqref{E2bb} imply that 
\beq{extra}{
\text{with probability $1-o(n^{-2})$, no path contributing to $E_3$ or $E_4$ has more than $\k_1$ edges.}
} 

We let $Z_3$ denote the number of edges $e=\set{A,B}$ that satisfy
\begin{enumerate}[(i)]
\item $A,B$ are blue.
\item $r_\e\leq |A-B|\leq 2R_\e$ and $|Y(i_{A,B},A)-A|\geq \e|A-B|$.. 
\item $e$ is on an induced path in $\cX_p$ that has length at least $(1+\e)|A-B|$ and at most $\k_1$ edges, each of length at most $R_\e$.
\end{enumerate}
Similarly, let $Z_4$ denote the number of edges $e=\set{A,B}$ that satisfy
\begin{enumerate}[(i)]
\item $A,B$ are blue.
\item $r_\e\leq |A-B|\leq 2R_\e$. 
\item $e$ is on an induced path in $\cX_p$ that has length at most $(1+\e)|A-B|$ and at most $\k_1$ edges, each of length at most $R_\e$.
\end{enumerate}
Let $Z_i',i=3,4$ be defined as for $Z_i$, without (i). Note that Lemma's \ref{longpaths} and \ref{E3} estimate $|E_i|$  through $|E_i|\leq Z_i'$ and showing $\E(Z_i')=O(n)$. Furthermore, $Z_i=Z_i',i=3,4$ if $\cU,\cB$ (see Remark \ref{rem+}) occur and these two events occur with probability $1-o(n^{-10})$. Thus we have for $i=3,4$,
\[
|E_i|\leq Z_i,\text{ w.h.p.}
\]
and
\[
 E(Z_i)\leq \E(Z_i'\mid \cB\cap\cU)\Pr(\cB\cap\cU)+n^2\Pr(\neg\cB\vee\neg\cU)\leq \E(Z_i')+n^2\Pr(\neg\cB\vee\neg\cU)=O(n).
\]
We will therefore bound the probability that either $Z_3$ or $Z_4$ exceeds its mean by $n$. We let $W=Z_3+Z_4$. To apply McDiarmid's Inequality we have to establish a Liptschitz bound for $W$. Our probability space consists of $\vartimes_{i=1}^{m^2}\Omega_i\times \vartimes_{C_j\sim C_k}\Omega_{j,k}$ where $\Omega_i$ is a set of at most $\r_0$ random points in subsquare $C_i$ together with a list of all of the edges inside $C_i$. We say that $C_j\sim C_k$ if there boundaries share a common point. Thus for a fixed $C_j$ there are usually 8 subsquares $C_k$ such that $C_j\sim C_k$. The set $\Omega_{j,k}$ determines the edges between points in $C_j$ and $C_k$. It can be represented by a $\r_0 \times \r_0$ $\set{0,1}$-matrix in which each entry appears independently with probability $p$. All in all there are $n^{1-o(1)}$ components of this probability space.
%\begin{enumerate}[Property P1]

A point $X\in\cX$ is in at most $\n_0=(9\r_0)^{\k_1}=n^{o(1)}$ of the paths counted by $W$. So, changing an $\Omega_i$ or an $\Omega_{i,j}$ can only change $W$ by at most $\n_1=2\r_0\n_0\k_1=n^{o(1)}$ and so the random variable $W$ is $\n_1$-Liptschitz.. 
It then follows from McDiarmid's inequality that
\[
\Pr(W\geq \E(W)+n)\leq \exp\set{-\frac{n^2}{2n^{1-o(1)}\n_1^2}}=e^{-n^{1-o(1)}}.
\]
This completes the proof of Theorem \ref{th2}.
\section{Proof of Theorem \ref{th3}}
For this we only have to observe that w.h.p. $K(X,i)$ exists for all $X,i$. This follows from the Chernoff bounds and the fact that the expected number of vertices in $K(X,i)$ grows faster than $\log n$. We can therefore use Lemma \ref{sp} to prove the existence of the required spanner.
\section{Summary and open questions}
There is a significant gap between the upper and lower bounds of Theorems \ref{th2} and \ref{th2a}, in their dependence on $\e,p$. Closing this gap is our greatest interest. 

We have considered a Euclidean version, asking for a $(1+\e)$-spanner and random geometric graphs. We could probably extend the results of Theorems \ref{th2}, \ref{th2a},\ref{th3} to $[0,1]^d,d\geq 3$. This does not seem difficiult. There is a slight problem in that the cones $K(i,X)$ intersect in sets of positive volume. The intersection volumes are relatively small and so the problems should be minor. We do not claim to have done this.

\end{document}